\documentclass[english, 11pt]{amsart}

\usepackage[english]{babel}
\usepackage[usenames,dvipsnames,svgnames,table]{xcolor}
\usepackage{amsfonts}
\usepackage{enumerate}
\usepackage{amsthm,amsmath,amssymb}
\usepackage[margin=3cm]{geometry}
\usepackage[normalem]{ulem}

\newtheorem{theorem}{Theorem}[section]
\newtheorem*{theorem*}{Theorem}
\newtheorem{thm}[theorem]{Theorem}
\newtheorem{proposition}[theorem]{Proposition}
\newtheorem*{prop*}{Proposition}
\newtheorem*{lemma*}{Lemma}
\newtheorem{prop}[theorem]{Proposition}
\newtheorem{corollary}[theorem]{Corollary}

\newtheorem{lemma}[theorem]{Lemma}

\newtheorem{lem}[theorem]{Lemma}

\newtheorem{def-theorem}[theorem]{Theorem-Definition}
\newtheorem*{statement*}{Statement}

\newtheorem{notation}[theorem]{Notation}
\def\thm{\begin{theorem}}
\def\ethm{\end{theorem}}
\def\lmm{\begin{lemma}}
\def\elmm{\end{lemma}}

\def\kor{\begin{corollary}}
\def\ekor{\end{corollary}}
\def\frag{\begin{question}\upshape}
\def\efrag{\end{question}}
\def\prop{\begin{proposition}}
\def\eprop{\end{proposition}}

\theoremstyle{definition}
\newtheorem{remark}[theorem]{Remark}
\newtheorem*{thm*}{Theorem}
\newtheorem{question}{Question}

\newtheorem{definition}[theorem]{Definition}

\newtheorem{innercustomthm}{Theorem}
\newenvironment{customthm}[1]
  {\renewcommand\theinnercustomthm{#1}\innercustomthm}
  {\endinnercustomthm}

\renewcommand{\subset}{\subseteq}
\newcommand{\Q}{\mathbb{Q}}

\newcommand{\N}{\mathbb{N}}

\newcommand{\R}{\mathbb{R}}

\newcommand{\Lm}{\mathcal{L}}
\newcommand{\cL}{\Lm}

\newcommand{\ord}{\mathrm{ord}}

\newcommand{\dd}{\text{$\delta$-$\dim$}}
\newcommand{\dc}{cl\text{-$\dim$}}

\newcommand{\Def}{\mathrm{Def}}

\newcommand{\U}{{ \mathbb U}}
\newcommand{\G}{{ \mathcal G}}
\newcommand{\fork}[1][]{%
  \mathrel{
    \mathop{
      \vcenter{
        \hbox{\oalign{\noalign{\kern-.3ex}\hfil$\vert$\rlap{\small{$^1$}}\hfil\cr
              \noalign{\kern-.7ex}
              $\smile$\cr\noalign{\kern-.3ex}}}
      }
    }\displaylimits_{#1}
  }
}
\newcommand{\afork}[1][]{%
  \mathrel{
    \mathop{
      \vcenter{
        \hbox{\oalign{\noalign{\kern-.3ex}\hfil$\vert$\rlap{\small{$^{2}$}}\hfil\cr
              \noalign{\kern-.7ex}
              $\smile$\cr\noalign{\kern-.3ex}}}
      }
    }\displaylimits_{#1}
  }
}
\newcommand{\nfork}[1][]{%
  \mathrel{
    \mathop{
      \vcenter{
        \hbox{\oalign{\noalign{\kern-.3ex}\hfil$\vert$\rlap{}\hfil\cr
              \noalign{\kern-.7ex}
              $\smile$\cr\noalign{\kern-.3ex}}}
      }
    }\displaylimits_{#1}
  }
}

\title[Strong density of definable types and closed ordered differential fields]{Strong density of definable types and closed ordered differential fields}

\author[Q. Brouette]{Quentin Brouette}
\address{\hskip-\parindent
Quentin Brouette\\
D\'epartement
  de Math\'ematique (Le Pentagone)\\
  Universit\'e de Mons\\
  20 place du Parc\\ B-7000 Mons\\ Belgium, Belgium.}
\email {quentin.brouette@umons.ac.be}

\author[P. Cubides Kovascics]{Pablo Cubides Kovacsics$^{\ast}$}
\address{\hskip-\parindent
Pablo Cubides Kovacsics \\ Laboratoire de math\'ematiques Nicolas Oresme\\ Universit\'e de Caen\\CNRS UMR 6139 
Universit\'e de Caen BP 5186\\
14032 Caen cedex, France. }
\email {pablo.cubides@unicaen.fr}
\thanks{${}^{\ast}$Supported by the ERC project TOSSIBERG (Grant Agreement 637027)}

\author[F. Point]{Fran\c coise Point$^{\dagger}$
}
\address{\hskip-\parindent
Fran\c coise Point\\
D\'epartement
  de Math\'ematique (Le Pentagone)\\
  Universit\'e de Mons\\
  20 place du Parc\\ B-7000 Mons\\ Belgium, Belgium.}
\email {point@math.univ-paris-diderot.fr}
\thanks{${\;}^{\dagger}$ Research Director at the FRS-FNRS, 
this material is based upon work supported by the National Science Foundation under Grant No. DMS-1440140 while the author was in residence at the Mathematical Sciences Research Institute in Berkeley, California, during the Spring 2017.}

\begin{document}

\maketitle

\begin{abstract} 
The following strong form of density of definable types is introduced for theories $T$ admitting a fibered dimension function $d$: given a model $M$ of $T$ and a definable set $X\subseteq M^n$, there is a definable type $p$ in $X$, definable over a code for $X$ and of the same $d$-dimension as $X$. Both o-minimal theories and the theory of closed ordered differential fields (CODF) are shown to have this property. As an application, we derive a new proof of elimination of imaginaries for CODF.
 \end{abstract}

\let\thefootnote\relax

\section*{Introduction}

After Hrushovski's abstract criterion of elimination of imaginaries was introduced in \cite{hrushovski2014}, density properties of definable types have drawn more and more attention (see \cite{johnsonthesis}, \cite{rideauthesis} and \cite{HKR2016}). Given a complete theory $T$ and a model $M$ of $T$, the property which often links imaginaries and definable types is the following: for every definable set $X\subseteq M^n$, there is a definable type $p\in S_n^T(M)$ in $X$ (i.e.,  $p$ contains a formula defining $X$) which is moreover definable over $acl^{eq}(e)$ for $e$ a \emph{code for $X$}. Here, a code for $X$ is an element of $M^{eq}$ which is fixed by all automorphisms (of a sufficiently saturated extension) fixing $X$ setwise. We say in this case that definable types are \emph{dense over $a$-codes}, and if the type $p$ can be taken to be definable over a code of $X$ without passing to the algebraic closure, we say that definable types are \emph{dense over codes}.  

For theories $T$ admitting a fibered dimension function $d$, we strengthen the above density properties as follows: for every definable set $X\subseteq M^n$, there is a definable type $p$ in $X$ which is definable over a code of $X$ (resp. over the algebraic closure of a code of $X$) and has the same $d$-dimension of $X$. We say in this case that definable types are \emph{$d$-dense over codes} (resp. \emph{over $a$-codes}). Very much inspired by Johnson's presentation of Hrushovski's abstract criterion for elimination of imaginaries \cite[Theorem 6.3.1]{johnsonthesis}, we obtain the following Proposition (later Proposition \ref{lem:1-to-n}):

\begin{prop*} 
	Let $T$ be a complete theory. 
	If definable $1$-types are dense over codes (resp. over $a$-codes) 
	then definable $n$-types are dense over codes (resp. $a$-codes). 
	Moreover, assuming $T$ admits a code-definable fibered dimension function $d$, 
	if $1$-types are $d$-dense over codes (resp. $a$-codes) 
	then so are all definable $n$-types.
\end{prop*}

The first part of the previous proposition is contained in \cite{johnsonthesis}. The result for $d$-density is new.  We obtain as a corollary that in any o-minimal theory, definable types are $\dim$-dense over codes (Proposition \ref{prop:strong-density-ominimal}), where $\dim$ denotes the usual topological dimension. We use this for the particular case of real closed fields in order to show that definable types are dense over codes for closed ordered differential fields (CODF), the main motivation of this paper. 

\begin{customthm}{A}\label{thm:densityoftypes} 
	Let $K$ be a model of CODF, $X\subseteq K^n$ be a non-empty definable set and $C$ be a code for $X$. 
	Then there is a $C$-definable type $p\in S_n^{CODF}(K)$ in $X$. 
\end{customthm}

We apply then Theorem A to give a new proof of elimination of imaginaries in CODF (see Theorem \ref{criterion}), very much inspired by the recent proofs given for algebraically closed valued fields by Hrushovski \cite{hrushovski2014} and further simplified by Johnson \cite[Chapter 6]{johnsonthesis}. 

Finally, using that models of CODF can be endowed with a fibered dimension function $\delta$-$\dim$, we strengthen Theorem \ref{thm:densityoftypes} by showing that definable types are $\delta$-$\dim$-dense over codes for CODF.

\begin{customthm}{B}\label{thm:strongdensityoftypes} 
	Let $K$ be a model of CODF, $X\subseteq K^n$ be a non-empty definable set and $C$ be a code for $X$.
	Then there is a $C$-definable type $p\in S_n^{CODF}(K)$ in $X$ such that $\delta$-$\dim(p)=\delta$-$\dim(X)$. 
\end{customthm}

The reason why we state Theorems \ref{thm:densityoftypes} and \ref{thm:strongdensityoftypes} separately is simply to stress that the new proof of elimination of imaginaries for CODF does not use the $\delta$-dimension (and the stronger density property of definable types). 

Motivated (in part) by the fact that in CODF, isolated types are not dense, the first author had shown in his thesis that definable types are dense (see \cite{brouettethesis} and \cite{brouette2017}). A key difference with the present work is that we carefully take into account the parameters over which types are defined. As already pointed out, this is necessary if, for instance, one is aiming to obtain elimination of imaginaries.

Even though CODF is NIP, it is not strongly dependent (hence not dp-minimal) \cite{brouettethesis}. We would like to point out a connection of the above proposition (Proposition \ref{lem:1-to-n}) with a result of Simon and Starchenko on definable types in dp-minimal theories \cite{simon2014}. Let us briefly state their result. Let $T$ be a dp-minimal theory, $M$ be a model of $T$ and $A$ be a subset of $M$. Simon and Starchenko show that, under the hypothesis that every unary $A$-definable set $X$ contains an $A$-definable type $p\in S_1^T(A)$, every non-forking formula $\varphi(x)$ (in possibly many variables) can be extended to a definable type over a model of $T$. Notice that assuming elimination of imaginaries, their hypothesis is equivalent to the assumption of the first part of Proposition \ref{lem:1-to-n}. Nevertheless, their result states the existence of a definable type extending a non-forking formula, without specifying the parameters over which such a type is definable. It is a natural question to ask whether a direct analogue of Proposition \ref{lem:1-to-n} holds for dp-minimal theories. 

Finally, we expect the strategy followed in this paper can be applied to other theories. In particular, we would like to apply similar techniques in future work to other topological structures endowed with a {\it generic derivation}, as defined in \cite{guzy-point2010}.

\

The article is laid out as follows. In Section \ref{sec:prelim}, we set the background and fix some terminology. 
In particular, we introduce what we call a code-definable fibered dimension function. Section 2 is divided in two parts. Subsection 2.1 is devoted to the proof of Proposition \ref{lem:1-to-n} and subsection 2.2 to the proof of Theorem \ref{thm:densityoftypes}. The alternative proof of elimination of imaginaries for CODF is presented in Section \ref{sec:EI}. 
In Section \ref{delta-dim}, we first recall the fibered dimension function $\delta$-dim that was defined in models of CODF and show that $\delta$-dim is code-definable. We end by proving that definable types in CODF are $\delta$-dim-dense over codes, that is, Theorem \ref{thm:strongdensityoftypes}.

\section{Preliminaries}\label{sec:prelim}

Let $\cL$ be a language, $T$ be a complete $\cL$-theory and $\U$ be a monster model of $T$, namely a model of $T$, 
$\kappa$-saturated and homogeneous for a sufficiently large cardinal $\kappa$. 
A subset of $\U$ is small if it is of cardinality strictly less than $\kappa$. Throughout, all substructures and subsets of $\U$ under consideration will be assume to be small. 

Let $A\subseteq \U$. Given an $\cL$-formula $\varphi(x)$, we let $\ell(x)$ denote the length of the tuple $x$. By an $\cL$-definable set, we mean a set defined by an $\cL$-formula with parameters in $\U$. When $\cL$ is clear from the context, we also use `definable' instead of `$\cL$-definable' and `$A$-definable' to specify that the parameters come from $A$. 
We denote by $\Def(\U)$ the set of all $\cL$-definable sets in $\U$. For integers $n\geq k\geq 1$, we denote by $\pi_k:\U^n\to \U^k$ the projection onto the first $k$ coordinates. Given a subset $X\subseteq \U^{n+m}$ and $a\in \pi_n(X)$, we let $X_a$ denote the fiber of $X$ over $a$, that is, $X_a:=\{x\in \U^m: (a,x)\in X\}$. 

By $S_n^T(A)$, we denote the Stone space of complete $n$-types consistent with $T$ with parameters in $A$. Let $B\subseteq \U$. Recall that a type $p\in S_n^T(A)$ is $B$-definable if for every $\cL$-formula $\varphi(x,y)$ with $x=(x_1,\ldots, x_n)$, $y=(y_1,\ldots, y_m)$ and without parameters, there is an $\cL$-formula $\psi(y)$ with parameters in $B$ such that for all $a\in A^m$, 
$\varphi(x,a)\in p$ if and only if $\U\models \psi(a)$. That formula $\psi(y)$ is classically denoted by $d_{p} x\varphi(x,y)$.

\subsection{Imaginaries and codes}\label{subsec:codes}

Let $X\subseteq \U^n$ be a definable set. 
A finite tuple $e\in \U^{eq}$ is a \emph{code for $X$} 
if for every $\sigma\in Aut_\cL(\U)$, it holds that $\sigma(X)=X$ if and only if $\sigma(e)=e$. 
Every non-empty $\cL$-definable set $X$ has at least one code and any two codes for $X$ are interdefinable. 
We let $c(X)$ denote \emph{some} code for $X$. 
Let $A$ be a subset of $\U$ and $p\in S^T_n(A)$, 
a \emph{code for $p$} is a (possibly infinite) tuple $e$ in $\U^{eq}$ such that for every $\sigma\in Aut_{\cL}(\U)$,
it holds that $\sigma(p)=p$ if and only if $\sigma(e)=e$. 
Every definable type $p$ has at least one code, namely, a tuple consisting of the elements 
$$\{c(d_{p} x\varphi(x,\U)): \varphi(x,y) \;\text{an $\cL$-formula} \}.$$ 
Classically (see for instance \cite{tent-ziegler2012} or \cite{messmer1996}), 
one sometimes calls \emph{canonical parameter} what we call here a code for a definable set 
and \emph{canonical base} what we call here a code for a type. 
Notice that any definable set $X$ is $c(X)$-definable.
When $T$ has elimination of imaginaries, each code $e$ is interdefinable with a finite tuple of $\U$ and, abusing of notation, we also denote such a finite tuple by $c(X)$. 

\subsection{Definable types} 

We gather three well-known properties on definable types whose proofs are left to the reader. 
\begin{lemma}\label{lem:dcltype} 
	Let $A\subseteq M\prec \U$ and $a\in \U^n$ be such that $tp(a/M)$ is $A$-definable. 
	If $b\in dcl(A,a)$, then $tp(a,b/M)$ is also $A$-definable.  
\end{lemma}

\begin{lemma}[{\cite[Lemma 6.2.7]{johnsonthesis}}] \label{transitivity} 
	Let $C\subset B\subset \U$ and let $a,b$ be two (possibly infinite) tuples of elements in $\U$. 
	If $tp(a/B)$ is $C$-definable and $tp(b/Ba)$ is $Ca$-definable, then $tp(a, b/B)$ is $C$-definable. 
	
\end{lemma}

We define ternary relations $a\fork[C] B$ and $a\afork[C] B$ by
\[
a\fork[C] B \Leftrightarrow \text{there exists a model $M\supseteq BC$ such that $tp(a/M)$ is $C$-definable}, 
\]
and $a\afork[C] B$ if and only if $a\fork[D] B$ where $D=acl^{eq}(C)$ (see also \cite[Lemma-Definition 6.2.8]{johnsonthesis}). The following is a standard consequence of Lemma \ref{transitivity} (see {\cite[Lemma 6.2.9]{johnsonthesis}} and {\cite[Lemma 6.2.12]{johnsonthesis}}). 

\begin{lemma}\label{lem:transitivity} Let $\nfork$ denote either $\fork$ or $\afork$. Assume that $a\nfork[C]B$ and $b\nfork[Ca]B$, then $ab\nfork[C]B$.  
\end{lemma}

\subsection{Dimension functions} We recall the definition of a fibered dimension function from \cite{vandendries1989}. 
\begin{definition}\label{def:dim} 
	A fibered dimension function on $\U$ is a function $d:\Def(\U)\to \mathbb{N}\cup\{-\infty\}$ 
	satisfying the following conditions for $X, Y\in \Def(\U)$:
	\begin{enumerate}[({Dim} 1)]
		\item $d(X) = -\infty$ iff $X = \emptyset$, $d(\{a\})=0$ for each $a\in \U$ and $d(\U)= 1$; 
		\item $d(X\cup Y)=\max\{d(X), d(Y)\}$;
		\item if $X\subseteq \U^n$, then $d(X)=d(X^\sigma)$ for each permutation $\sigma$ of $\{1,\ldots,n\}$;  
		\item if $X\subseteq \U^{n+1}$ and $l\in\{0,1\}$, then the set $X(l):=\{a\in \pi_n(X): d(X_a)=l\}$ belongs to $\Def(\U)$ 
		and $d(\{(x,y)\in X : x\in X(l)\})=d(X(l))+l$.
	\end{enumerate}
\end{definition}
If in addition, $d$ satisfies the property (Dim 5) below, we will say that $d$ is \emph{code-definable}.
\begin{enumerate}[({Dim} 5)]
	\item If $X\subseteq \U^{n+1}$, the sets $X(0)$ and $X(1)$ are $c(X)$-definable.
\end{enumerate}

\

Let $d$ be a dimension function on $\U$. Given a set $C\subseteq \U$, we extend the function $d$ to the spaces of types $S_n^T(C)$ by setting, for any $p\in S_n^T(C)$, 
\begin{equation}\label{eq:type-dimension}\tag{D}
d(p):=\inf\{d(\varphi(\U,c)) : \varphi(x,c) \in p\}.
\end{equation}	
Note that there always exists a formula $\varphi(x,c)\in p$ such that $d(p)=d(\varphi(\U,c))$. The converse is shown in Proposition \ref{real}. For a tuple $a\in \U^n$, we set $d(a/C)$ as $d(tp(a/C))$.

Here are some standard consequences of Definition \ref{def:dim} whose proofs are left to the reader. 

\begin{lemma}\label{lem:dim} Let $X$ and $Y$ be definable sets. 
\begin{enumerate}
\item if $X\subseteq Y$, then $d(X)\leqslant d(Y)$; 
\item if $X$ is non-empty and $X\subseteq \U^{n+1}$, then $\pi_n(X)=X(1)\cup X(0)$;
\item for any $C\subset \U$, $a\in \U^n$ and $b\in \U$, we have that $d(a/C)\leqslant d(a,b/C)$;
\item for any $C\subseteq D\subset \U$ and $a\in \U^n$, we have that $d(a/D)\leqslant d(a/C)$.  
\end{enumerate}
\end{lemma}

\begin{lemma}\label{lem:additivity} 
	For any $C\subset \U$, $a\in \U^n$ and $b\in \U$, we have that $d(a,b/C)=d(a/C)+d(b/Ca)$. 
\end{lemma} 

\begin{proof}
Let $X\subseteq \U^{n+1}$ be a $C$-definable set such that $(a,b)\in X$ and $d(X)=d(a,b/C)$. 
Let $Z\subseteq \U^n$ be such that $a\in Z$ and $d(a/C)=d(Z)$. 
After possibly replacing $X$ by the subset $X':=\{(x,y)\in X: x\in Z\}$, 
we may further assume that $d(\pi_n(X))=d(a/C)$, keeping the property that $d(X)=d(a,b/C)$.
We split in cases depending on the value $d(b/Ca)$. 

\textit{Case 1:} Suppose that $d(b/Ca)=0$. By possibly restricting $X$ further (adding a formula witnessing that $d(b/Ca)=0$), we may assume that $a\in X(0)$. 
So $d(a/C)= d(X(0))=d(\pi_n(X))$. For $Y:=\{(x,y)\in X: x\in X(0)\}$,
we have that $(a,b)\in Y$ and that $d(a,b/C)\leqslant d(Y)=d(X(0))=d(a/C)$. 
Part $(3)$ of Lemma \ref{lem:dim} implies the equality. 

\textit{Case 2:} Suppose that $d(b/Ca)=1$.
By the case assumption, $a\in X(1)$, and since $d(a/C)=d(\pi_n(X))$, we have $d(a/C)=d(X(1))$. 
Therefore, $d(a,b/C)=d(X)=d(X(1))+1=d(a/C)+1$.
\end{proof}

\begin{proposition}\label{real}
Let $X\subset \U^n$ be a definable set with parameters in $C$. Then there exists a tuple $a\in \U^n$ such that $d(a/C)=d(X)$. 
\end{proposition}

\begin{proof}
We proceed by induction on $n$ and use the fact that the dimension is fibered. Suppose $n=1$.

If $d(X)=0$, then $X\neq \emptyset$, so take $a\in X$. By definition $d(tp(a/C))=0$.

If $d(X)= 1$, observe that for any formula $\varphi$ over $C$, we have $X=(\varphi\wedge X)\vee (\neg\varphi\wedge X)$. By (Dim 2), one of (or both) $\varphi$ or $\neg\varphi$ is of dimension $1$. Consider the following partial type over $C$, 
$p:=\{\varphi:\; \varphi(\U)\subset X, d(\neg\varphi(\U))<d(X),\;\;\text{$\varphi$ with parameters in $C$}\}$.
By (Dim 2), $p$ is finitely consistent with $X$. 
Let $a$ be a realisation of $p$, then $d(a/C)=d(X)$. Suppose that $d(a/C)=0$, then there would exist $\psi$ a formula with parameters in $C$ such that $d(\psi)=0$ and $\psi(a)$ holds. This implies that $d(\neg\psi\wedge X)=1$, and so $\neg\psi\wedge X\in p$, a contradiction.

Suppose now $X\subset \U^n$ with $n>1$ and $d(X)=d$. We have that $d(X)=d(X(i))+i$ for some $i\in\{0,1\}$. By induction hypothesis there is an tuple $b\in X(i)$ such that $d(b/C)=d(X(i))$. Since $d(X_b)=i$, by the case $n=1$, there is $b_0\in X_{b}$ with $d(b_0/C, b)=i$. Therefore, by Lemma \ref{lem:additivity}, 
\[
d(X)=d(X(i))+i=d(b/C)+d(b_0/C, b)=d(b, b_0/C).
\]
\end{proof}

\subsection{O-minimal theories and dimension}\label{o-minimal}

Let $\cL$ be a language containing the order relation $<$ and let $T$ be any o-minimal theory (we will always assume that $<$ is dense). Our main example will be the theory RCF of real-closed fields, which we regard as a theory in the language of ordered rings $\cL_{or}:=\{<,+,\cdot,0,1\}$. Let $\U$ be a monster model of $T$. The topological closure of a subset $X\subseteq \U^n$ is written $\overline{X}$ and its interior $Int(X)$.  Notice that both $\overline{X}$ and $Int(X)$ are $c(X)$-definable. 

For a definable set $X\subseteq \U^n$, we denote by $\dim(X)$ the dimension of $X$. 
It corresponds to the biggest integer $k\leq n$ for which there is a coordinate projection $\pi:\U^n\to\U^k$ such that $Int(\pi(X))\neq\emptyset$. This dimension function on $\U$ is fibered  \cite[Chapter 4]{vandendries1998}, namely it satisfies the first four properties of Definition \ref{def:dim}. 
Moreover since one can express that a subset of $\U$ is of dimension $1$, the function $\dim$ satisfies (Dim 5). 
Indeed, for a definable set $X\subseteq \U^{n+1}$, we have that
\[
X(1)=\{x\in \pi_n(X): \exists a\exists b\;\; [a<b\, \wedge \,(a,b)\subseteq X_x]\}.
\]
So, $X(1)$ is $c(X)$-definable. 
Since $X(0)=\pi_n(X)\setminus X(1)$, we have that $X(0)$ is $c(X)$-definable as well. 

Let $A\subset \U$ and $a\in \U$. Since $T$ is o-minimal, $tp(a/A)$ is determined by the quantifier-free order type of $a$ over $A$. 
We say that $z$ realizes the type $a^+$ over $A$ whenever $z$ is a realization of the $a$-definable type $\{(a<x<c) :c\in A\text{ and }c>a\}$. 
Analogously, we say that $z$ realizes the type $-\infty$ over $A$ whenever
$z$ is a realization of the $\emptyset$-definable type $\{(x<b) : b\in A\}$.

\begin{lemma}\label{lem:polynomial} Let $K$ be a real closed field and $Y\subseteq K^n$ be a semi-algebraic set. Suppose that $\dim(Y)<n$. Then there is a non-trivial polynomial $P\in K[x]$ with $x=(x_1,\ldots,x_n)$ such that $Y\subseteq \{x\in K^n: P(x)=0\}$. 
\end{lemma}
\begin{proof}
Since $Y$ is semi-algebraic, it is a finite union of non-empty sets $Y_1\cup\ldots\cup Y_m$ of the form 
\[
Y_i:=\{a\in K^n: \bigwedge_{j=1}^{n_i} Q_{ij}(a)>0 \wedge P_i(a)=0\}, \text{ with $Q_{ij}, P_i\in K[x]$.} 
\]
Suppose that for some $i$ , the polynomial $P_i$ is the zero polynomial. Then $Y_i$ would be open and so $\dim(Y_i)=n$ (and so $\dim(Y)=n$) a contradiction.
So we take $P(x)=\prod_{i=1}^\ell P_i(x)$, and the set $Y$ is contained in the zero locus of $P$. 
\end{proof}

\subsection{Closed ordered differential fields}

Hereafter we let $\cL_\delta:=\cL_{or}\cup\{\delta\}$ be the language of ordered differential rings, where $\delta$ is a unary function symbol. Let $\cL_{or}^+:=\cL_{or}\cup \{{\;}^{-1}\}$ be the language of ordered fields and $\cL_\delta^+:=\cL_{or}^+\cup\{\delta\}$ the  	language of ordered differential fields. The $\cL_\delta^+$-theory of ordered differential fields is simply the $\cL_{or}^+$-theory of ordered fields together with the axioms for a derivation $\delta$ and CODF is its model completion \cite[section 2]{singer1978}. Since CODF is the model completion of a universal theory, it admits quantifier elimination in $\cL_\delta^+$ and it is easy to check that it also admits quantifier elimination in $\cL_\delta$.

Let $\U$ be a monster model of CODF. 
In particular, $\U$ is also a monster model of RCF. 
We equip $\U$ with the order topology. Given $a,b\in \U^n$ and $\epsilon\in \U$ with $\epsilon>0$, we abbreviate by $\vert a-b\vert<\epsilon$
the formula $\bigwedge_{i=1}^n \vert a_i-b_i\vert <\epsilon$.

For a subfield $F\subseteq \U$ and a subset $A\subseteq \U$, the smallest real closed field in $\U$ containing both $F$ and $A$ is denoted by $F(A)^{rc}$. When $F$ is a differential field, the smallest differential subfield of $\U$ containing both $F$ and $A$ is denoted by $F\langle A\rangle$. Given $a=(a_1,\ldots,a_n)\in \U^n$ we let $F(a)^{rc}$ (resp. $F\langle a\rangle$) denote $F(\{a_1,\ldots,a_n\})^{rc}$ (resp. $F\langle\{a_1,\ldots,a_n\}\rangle$). 

For $n\geqslant 0$ and $a\in \U$, we define 
\[
\delta^n(a):=\underbrace{\delta\circ\cdots\circ\delta}_{n \text{ times}}(a), \text{ with $\delta^0(a):=a$,}
\]
and $\bar{\delta}^n(a)$ as the finite sequence $(\delta^0(a),\delta(a),\ldots,\delta^n(a))\in \U^{n+1}$. We let $\bar{\delta}(a)$ denote the infinite sequence $(\delta^n(a))_{n\geqslant 0}$. 
We will denote by $K\{x\}$ the differential ring of differential polynomials (in one differential indeterminate $x$) over $K$. Abusing of notation, we identify $K\{x\}$ with the ordinary polynomial ring $K[\delta^j(x): j\in \N]$ in indeterminates $\delta^j(x)$, 
endowed with the natural derivation extending the one of $K$ with the convention that $\delta^0(x)=x$ and $\delta(\delta^j(x))=\delta^{j+1}(x)$. 
The order of a differential polynomial $f(x)\in K\{x\}$, denoted by $\ord(f)$, is the smallest integer $n\geqslant 0$ such that there is a polynomial $F\in K[\bar{\delta}^n(x)]$ such that $f(x)=F(\bar{\delta}^n(x))$. 
The (algebraic) polynomial $F$ is unique and will be denoted hereafter by $f^*$. 
Below, $\frac{\partial}{\partial \delta^{n}(x)}$ denotes the usual derivative of polynomials 
in $K[\delta^j(x): j\in \N]$ with respect to the variable $\delta^n(x)$.

Let us recall the axiomatisation of CODF given in \cite{singer1978}. An ordered differential field $K$ is a model of CODF if it is real-closed and given any $f, g_1,\cdots,g_m\in K\{x\}\setminus \{0\}$, with $n:=\ord(f)\geqslant \ord(g_i)$ for all $1\leqslant i \leqslant m$, if there is $c\in K^{n+1}$ such that 
\[
f^*(c)=0 \wedge \left(\frac{\partial}{\partial \delta^{n}(x)}f^*\right)(c)\neq 0 \wedge \bigwedge_{i=1}^m g_i^*(c)>0,
\]
then there exists $a\in K$ such that $f(a)=0\wedge\bigwedge_i^m g_i(a)>0$. Equivalently \cite[Theorem 4.1, Example(s) 2.22]{guzy-point2010}, $K$ is a model of CODF if it is real-closed and given $f\in K\{x\}\setminus \{0\}$ with $n:=\ord(f)$ and $\epsilon >0$, if there exists $c\in K^{n+1}$ such that 
\[
f^*(c)=0\wedge \left(\frac{\partial}{\partial \delta^{n}(x)}f^*\right)(c)\neq 0,
\] 
then there exists $a$ such that $f(a)=0\wedge\vert \bar{\delta}^{n}(a)-c\vert<\epsilon$. 

\

Let $\cL_\delta^-=\{<,+,-,\cdot,\delta\}$ be the reduct of $\cL_\delta$ where the constants have been removed. Let $\varphi(x_1,\cdots,x_m)$ be a quantifier-free $\cL_\delta^-$-formula with $\ell(x_i)=1$ for all $i\in\{1,\ldots,m\}$. For each $i\in\{1,\ldots,m\}$, let $n_i\geqslant 0$ be the largest integer such that $\delta^n_i(x_i)$ is a term of $\varphi$. For $\overline{x}_i$ a tuple of variables with $\ell(\overline{x}_i)=n_i+1$, we write $\varphi^\ast$ for the $\cL_{or}$-formula $\varphi^*(\overline{x}_1,\cdots,\overline{x}_n)$ such that the $\cL_\delta^-$-formula $\varphi(x_1,\cdots,x_m)$ is $\varphi^*(\bar{\delta}^{n_1}(x_1),\cdots,\bar{\delta}^{n_m}(x_m))$. We extend this functor to quantifier-free $\cL_\delta$-formulas with parameters (constants treated as such) as follows. Let $\psi(x_1,\cdots,x_m)$ be a quantifier-free $\cL_\delta$-formula with parameters $c_1,\ldots,c_k$ in $\U$. Then $\psi$ is of the form $\varphi(x_1,\cdots,x_m,c_1,\ldots,c_k)$ with $\varphi(x_1,\cdots,x_m,z_1,\ldots,z_k)$ a quantifier-free $\cL_\delta^-$-formula without parameters. Define $\psi^*$ as the $\cL_{or}$-formula $\varphi^*(\overline{x}_1,\cdots,\overline{x}_m,\bar{c}_1,\ldots, \overline{c}_k)$ with $\bar{c}_i=\bar{\delta}^{\ell(\bar{z}_i)}(c_i)$, where $\bar{z}_i$ is the tuple of variables appearing in $\varphi^*(\overline{x}_1,\cdots,\overline{x}_m, \bar{z}_1,\cdots,\overline{z}_k)$. 
\

For any $B\subseteq \U$, the $\cL_\delta$-type of a tuple $a=(a_1,\ldots,a_n)\in \U^n$ over $A$ is denoted by $tp_\delta(a/B)$, 
whereas $tp(a/B)$ denotes its restriction to $\cL_{or}$. To distinguish between codes in $\cL_{or}$ and $\cL_\delta$, we use the following notational convention: given an $\cL_\delta$-definable set $X$, we let $c_\delta(X)$ denote a code for $X$ in $\cL_\delta$ and, for an $\cL_{or}$-definable set $X$, we let $c(X)$ denote a code for $X$ in $\cL_{or}$.

Throughout $K$ will always denote a submodel of CODF. A useful observation is the following lemma. 

\begin{lemma} \label{forgetful}
	Let $a\in \U$ and $B\subseteq K^{eq}$. 
	If for every integer $n\geqslant 0$ the type $tp(\bar{\delta}^n(a)/K)$ is $B$-definable, 
	then the type $tp_\delta(a/K)$ is $B$-definable. 
\end{lemma}

\begin{proof} By quantifier elimination it suffices to show that every quantifier free $\cL_\delta$-formula has a definition. 
	Let $y=(y_1,\ldots,y_m)$ and $\varphi(x,y)$ be a quantifier-free $\cL_\delta$-formula without parameters. Let $\ell$ be a sufficiently large integer such 
	that the formula $\varphi^\ast$ may be expressed as a formula $\varphi^\ast(x_0,\ldots,x_\ell,\bar{y})$ where 
	$\bar{y}=(\bar{y}_1,\ldots,\bar{y}_m)$ with $\bar{y}_i:=(y_{i1},\ldots,y_{i\ell})$. Since $tp(\bar{\delta}^\ell(a)/K)$ is $B$-definable, let $\psi(\bar{y})$ be an $\cL_{or}$-formula with parameters over $B$ such that for all $\bar{b}\in K^{m\ell}$
\[
\U\models \varphi^*(\bar{\delta}^\ell(a),\bar{b}) \Leftrightarrow \U\models \psi(\bar{b}).
\]
Define now the $\cL_\delta$-formula (with parameters in $B$) $\theta(y_1,\ldots,y_m)$ as $\psi(\bar{\delta}^\ell(y_1), \ldots, \bar{\delta}^\ell(y_m))$. We show that $\theta$ is a definition for $\varphi$. Indeed for every $u=(u_1,\ldots,u_m)\in K^m$
\[
\begin{array}{lll}
\U\models \varphi(a,u) 	& \Leftrightarrow & \U\models \varphi^*(\bar{\delta}^\ell(a),\bar{\delta}^\ell(u_1), \ldots, \bar{\delta}^\ell(u_m))\\
										& \Leftrightarrow & \U\models \psi(\bar{\delta}^\ell(u_1), \ldots, \bar{\delta}^\ell(u_m))\\																															& \Leftrightarrow & \U\models \theta(u_1,\ldots,u_m).
\end{array}
\]
\end{proof}

In models of CODF, we have the following folklore topological property. For the reader convenience, we give a proof below.
It is similar to the proof of \cite[Lemma 3.12 (1)]{guzy-point2010} and \cite[Corollary, 3.13]{guzy-point2010}. 

\begin{lem} [Density of primitives]\label{constants}
For any $b\in K$, any non-empty open subset $O$ of $K^{n}$ and any $n\in \N^{*}$, there is 
	$a\in K$ such that $\bar \delta^{n-1}(a)\in O$ and $\delta^{n}(a)=b$. 
	\end{lem}
\begin{proof}

Consider the differential polynomial $f(x):=\delta^{n}(x)-b$ and let $c\in O$. We have that
\[
f^*(c,b)=0 \wedge \left(\frac{\partial}{\partial \delta^{n}(x)}f^*\right)(c,b)=1\neq 0.
\]
Let $\epsilon>0$ in $K$ be such that the ball $\{x\in K^n : \vert x-c\vert<\epsilon\}\subseteq O$. By the axiomatization of CODF, there is an element $a\in K$ such that $f(a)=0$ and $\vert \bar{\delta}^n(a)-c\vert<\epsilon$. Thus $\delta^n(a)=b$ and $\bar\delta^{n-1}(a)\in O$.  
\end{proof}

Let us recall the following lemma about differential fields. 

\begin{lemma}[{\cite[Lemma 1.9]{marker1996}}]\label{lem:algind} 
	Let $a\in \U$.
	Suppose $\delta^{n+1}(a)$ is algebraic over $K(\bar{\delta}^n(a))$. 
	Then for all $k\geqslant n+1$, $\delta^k(a)$ is algebraic over $K(\bar{\delta}^n(a))$
	and also belongs to $K(\bar{\delta}^{n+1}(a))$. 
\end{lemma}

We will also need the following lemma about definable sets in CODF, which is essentially contained in \cite[Lemma 2.1]{point2011}.
The proof in \cite{point2011} uses semi-algebraic cell decomposition and that RCF has finite Skolem functions. 
That last property means that for every definable set $X\subseteq K^{n+1}$,
if for every $a\in \pi_n(X)$, $X_a$ is of finite cardinality $d$, then
it holds that  
there are definable functions $f_i:\pi_n(X)\to K$ (where $1\leq i \leq n$), such that $X_a=\bigcup_{i=1}^d f_i(a)$.

A similar result also holds in the broader context of topological differential fields given in \cite{guzy-point2010}. Both for completeness and the reader's convenience, we include here a proof for CODF (from which the more general proof can be easily extracted). 
We will use that RCF has a fibered dimension (which can be shown directly, without using cell decomposition) 
and finite Skolem functions. 
We also need the fact that for a definable set $X$, it holds that $\dim(X\setminus Int(X))<\dim X$ and 
the fact that on an open definable set, definable functions are discontinuous on a subset of smaller dimension.
 
\begin{notation} Let $A\subset \U$, then $A^{\nabla_{n}}:=\{\bar{\delta}^{n-1}(z):z\in A\}$ (if the context is clear we drop the index $n$ and simply use $A^{\nabla}$).
\end{notation}

\begin{lemma}\label{lem:density} 
	Let $X \subseteq K$ be a non-empty $\cL_\delta$-definable set defined by a quantifier-free formula $\varphi$ and let $n$ be the number of  free variables of the formula $\varphi^*$. 
	Then there is an $\cL_{or}$-definable set $X^{\circledast}\subseteq K^n$ such that $X^\nabla$
	is contained and dense in $X^{\circledast}$. In particular, $X$ is dense and contained in $\pi_1(X^{\circledast})$. 
\end{lemma}

\begin{proof}
Let $X^*$ denote the set $\varphi^*(K)$. Given any finite partition $\{X^*_i\}_{i\in I}$ of $X^*$ such that for each $i\in I$ the set $X^*_i$ is defined by an $\cL_{or}$-formula $\chi_i$, it is enough to show the result for each $\cL_\delta$-definable set $X\cap \pi_1(X^*_i)$. Indeed, since each set $X\cap\pi_1(X^*_i)$ is defined by the formula $
\theta_i(x):=\varphi(x)\wedge\chi_i(\bar{\delta}^{n-1}(x))$, notice that $\theta_i^*=\varphi^*\wedge \chi_i$ is equivalent to $\chi_i$ since $\chi_i\rightarrow \varphi^*$. So assuming the result for each $\theta_i(K)$, there are sets $\theta_i(K)^\circledast$ such that $\theta(K)^\nabla$ is contained and dense in $\theta_i(K)^\circledast$. Taking $X^\circledast=\bigcup_{i\in I} \theta_i(K)^\circledast$ shows the result for $X$. We will use this observation throughout the proof. 

\

For any $\cL_{\delta}$-definable subset $Y\subseteq X$, define an integer $e(Y^*)\geqslant 1$ by 
\[
e(Y^*):=
\begin{cases}
\min\{m : \dim(\pi_m(Y^*))<m\} & \text{ if such an $m\leqslant n$ exists}\\ 
n+1 & \text{otherwise.}\\
\end{cases}
\]

Let $1\leq k \leq n$. Note that $e(Y^*)=k+1$ implies that $\dim(\pi_m(Y^*))=k$.

We will show the result by induction on $e(X^*)$, for all $\cL_\delta$-definable sets $X$ simultaneously. To show the base case of the induction, suppose that $e(X^*)=1$. 
This implies that $\pi_1(X^*)$ is finite and therefore that $X$ (and $X^{\nabla}$) are finite too. Setting $X^\circledast:=X^\nabla$ satisfies all the requirements (as a finite set, it is certainly $\cL_{or}$-definable). Suppose that we have shown the result for all integers smaller than $m\leqslant n$ and that we have $e(X^*)=m+1$. We split in cases depending on whether $m=n$ or $m<n$. 

\

\textbf{Case 1:} Suppose that $e(X^*)=n+1$. This implies that $\dim(X^*)=n$. Set $X_1^*=Int(X^*)$ and $X_2^*:=X^*\setminus X_1^*$. Consider for $i=1,2$ the set $X_i:=X\cap \pi_1(X_i^*)$. By the starting observation, it suffices to show the result for $X_1$ and $X_2$. Since $\dim(X_2^*)<n$, $e(X_2^*)\leqslant n$, and the result follows for $X_2$ by induction. We show that for $X_1$ the set $X_1^\circledast:=X_1^*$ satisfies the result. 
The fact that $X_1^\nabla$ is contained in $X_1^*$ follows by assumption, so it suffices to show that $X_1^\nabla$ is dense in $X_1^*$. Pick $a\in X_1^*$ and $\epsilon>0$ in $K$. Since $X_1^*$ is open, let $\epsilon_0\leqslant \epsilon$ be such that $\{x\in K^n: |x-a|<\epsilon_0\}$ is contained in $X_1^*$. Then by Lemma \ref{constants}, there is $b\in K$ such that $|\bar{\delta}^{n-1}(b)-a|<\epsilon_0$. This shows that $\delta^{n-1}(b)\in X_1^*$, which implies that $b\in X_1$ and completes the proof of this case.  

\

\textbf{Case 2:} Suppose that $e(X^*)=m+1\leqslant n$. By assumption we have that $\dim(\pi_m(X^*))=m$ so $Int(\pi_m(X^*))\neq\emptyset$. By induction we may assume that $\pi_m(X^*)$ is open. Indeed, partition $X^*$ into $X_1^*$ and $X_2^*$, with $X_1^*:\{y\in X^*: \pi_m(y)\in Int(\pi_m(X^*))\}$ and $X_2^*:= X^*\setminus X_1^*$. Set again $X_i$ to be $X\cap \pi_1(X_i^*)$. We have that $\pi_m(X_1^*)$ is open, and the result follows by induction for $X_2$ since, as $\pi_m(X_2^*)$ has dimension strictly less than $m$, hence $e(X_2^*)\leqslant m$. 

By Lemma \ref{lem:polynomial}, there is a non-trivial polynomial $P(z,w)$ with coefficients in $K$ and $\ell(z)=m$ and $\ell(w)=1$, such that $\pi_{m+1}(X^*)$ is contained in the zero locus of $P$, which we denote by $Z$. For every $x\in \pi_m(X^*)$, there are at most $s$ elements in $Z_x$. By possibly partitioning $\pi_m(X^*)$ (and $X$) we may assume $Z_x$ has exactly $s$ elements. By the existence of finite Skolem functions, let $h_1,\ldots,h_s$ be definable functions such that for all $x\in \pi_m(X^*)$, $Z_x=\{h_1(x),\ldots, h_s(x)\}$. Therefore, by decomposing $X^*$ and $X$, we may suppose that $\pi_{m+1}(X^*)=\bigcup_{i=1}^s\{(x,y)\in \pi_{m}(X^*)\times K: h_i(x)=y\}$. Moreover, we may suppose that $h_i$ is continuous. This follows by induction and further partitioning $X^*$ and $X$ since the set $\{x\in \pi_m(X^*): h_i \text{ is discontinuous as $x$}\}$ is of strictly lower dimension than $\pi_m(X^*)$. 

Writing $P$ as a polynomial in the variable $w$, we obtain polynomials $P_k(z)$ such that $P(z,w)=\sum_{k=0}^d P_k(z)w^k$. We will show the result by induction on $d$; we will call this inductive step, the ``\emph{case induction}'', as opposed to the ``\emph{ambient induction}''. Note that the case $d=0$ never arises, since it would imply that $\pi_{m}(X^*)$ is included in the zero locus of $P_0(z)$, which contradicts that $\pi_m(X^*)$ is open. 
Express $\delta(P)=\sum_{k=0}^d \delta(P_k).w^k+\sum_{k=1}^d P_k.k.w^{k-1}.\delta(w)$. So for $\ell\geq 1$ and for a tuple of elements $(z,w)\in Z$, there is a polynomial $g_{\ell}$ in $z,\bar \delta^{\ell-1}(w)$ such that $\delta^{\ell}(w).(\sum_{k=0}^{d-1} P_{k+1}(z).(k+1).w^{k})=g_{\ell}(z,\bar \delta^{\ell-1}(w))$. We suppose the result for all $X^*$ such that $\pi_{m+1}(X^*)$ is included in the locus of a non zero polynomial of degree $<d$ in the last variable. So by the case induction, we may assume that for all $(z,w)\in \pi_{m+1}(X^*)$, we have $\sum_{k=0}^{d-1} P_{k+1}(x).(k+1).{w}^{k}\neq 0$.
Define $f_{0,i}(x):=h_i(x)$ and by induction on $\ell\geqslant 1$, define the function $f_{\ell,i}:\pi_m(X^*)\to K$ by
 $$f_{\ell,i}(x)=\frac{g_{\ell}(x,\bar f_{\ell-1,i}(x))}{\sum_{k=0}^d P_k(x).k.{h_i(x)}^{k-1}}$$ 
where $\bar{f}_{\ell-1,i}(x)=(f_{0,i}(x),\ldots,f_{\ell-1,i}(x))$.
Each function $f_{\ell,i}$ is well-defined by the case induction hypothesis and continuous everywhere except on a set of dimension strictly less than $m$. Therefore, by possibly partitioning $X^*$ and $X$ and applying our ambient induction, we may assume each function $f_{\ell,i}$ is well-defined and continuous at every point in $\pi_{m}(X^*)$. 
Define $Y_{i}^\circledast$ as follows:
\[
Y_{i}^\circledast:=\{(x,y_0,\ldots,y_{n-m-1})\in X^*: \bigwedge_{\ell=0}^{n-m-1} f_{\ell,i}(x)=y_{\ell}\}. 
\]  
Either $\pi_{m}(Y_{i}^\circledast)$ is an open subset of $\pi_{m}(X^*)$ or we consider on one hand $Int(\pi_{m}(Y_{i}^\circledast))$ and on the other hand $\pi_{m}(Y_{i}^\circledast)\setminus Int(\pi_{m}(Y_{i}^\circledast))$. We apply the induction hypothesis to that last subset.  

Set $X_{i}^\circledast:=\{(x,y_0,\ldots,y_{n-m-1})\in K^n: x\in Int(\pi_{m}(Y_{i}^\circledast))\wedge\bigwedge_{\ell=0}^{n-m-1} f_{\ell,i}(x)=y_{\ell}\}$

Finally define $X^\circledast:=\bigcup_{i=1}^s X_{i}^\circledast$.
Let us first show that $X^\nabla\subseteq X^\circledast$. Pick $\bar{\delta}^{n-1}(b)\in X^\nabla$. Since $\bar{\delta}^{m}(b) \in \pi_{m+1}(X^*)$, this implies that for some $i$, $\delta^{m+1}(b)\in X_{m,i}$. 
Therefore, by our construction of $f_{\ell,i}$, $\delta^{m+\ell}(b)=f_{\ell,i}(\bar{\delta}^{m+\ell-1}(b))$, for all $\ell\geqslant 1$. This shows that $\bar{\delta}^{n-1}(b)\in X^\circledast$. To show the density of $X^\nabla$ in $X^\circledast$, let $a=(a_0,\ldots,a_{n-1})\in X^\circledast$ and $\epsilon>0$ in $K$. Notice that $\pi_{m}(X^\circledast)$ is open. Let $\epsilon_0\leqslant \epsilon$ be such that $\{x\in K^{m}: |x-(a_0,\ldots,a_{m-1})|<\epsilon_0\}$ is contained in $\pi_{m}(X^\circledast)$. Since the functions $f_{0,i},\ldots,f_{n-m-1,i}$ are continuous, let $\delta\leqslant \epsilon_0$ be such that for all $\ell\in\{0,\ldots,n-m-1\}$ and all $x\in \pi_m(X^*)$
\begin{equation}\label{eq:continuity}
|x-(a_0,\ldots,a_{m-1})|<\delta\Rightarrow |f_{\ell,i}(x)-f_{\ell,i}((a_0,\ldots,a_{m-1}))|<\epsilon_0.
\end{equation}
By the axiomatisation of CODF, there is a point $b\in K$ such that $P(\bar{\delta}^{m-1}(b),\delta^m(b))=0$ and $|\bar{\delta}^{m}(b)-(a_0,\ldots,a_m)|<\delta$. So there is $1\leq i\leq s$ such that $\delta^m(b)=h_{i}(\bar{\delta}^{m-1}(b))$ and so by continuity using (\ref{eq:continuity}),  $f_{\ell,i}(\bar{\delta}^{m-1}(b))$ (which is equal to $\delta^{m+\ell}(b)$) will be close to $a_{m+\ell+1}$ for $0\leq \ell\leq n-m-1$.
This shows in particular that $\bar{\delta}^{m-1}(b)\in \pi_{m}(X^*)$ and that $\bar{\delta}^n(b)\in X_{i}^\circledast$. Therefore $b\in X$ and 
$|\bar{\delta}^n(b)-a|<\epsilon$.
\end{proof}

\section{Definable types and imaginaries}\label{sec:unary}

In this section, we show that in the Stone space of the theory CODF, definable types are dense over codes.
We start by proving that for any complete theory, 
it suffices to show that density result for definable $1$-types (instead of definable $n$-types for arbitrary $n$).

\subsection{Reducing to the case $n=1$} Let $T$ be a complete $\cL$-theory, $\U$ be a monster model of $T$ and $d$ be a fibered dimension function on $\U$. Let $K$ be a model of $T$.  

\begin{definition}\label{def:strong-dim-density} We say that definable $n$-types are \emph{dense over codes in $T$} (resp. dense over $a$-codes) if for every non-empty definable subset $X\subseteq K^n$ there is $a\in \U^n$ such that 
\begin{enumerate}[(1)]
\item $a\in X(\U)$,
\item $tp(a/K)$ is $c(X)$-definable (resp. $acl^{eq}(c(X))$-definable).
\end{enumerate}
If moreover we impose that $d(a/K)=d(X)$, we say that definable $n$-types are \emph{$d$-dense over codes in $T$} (resp. $d$-dense over $a$-codes in $T$).
\end{definition}

Notice that with this terminology, Theorem \ref{thm:densityoftypes} states that all definable $n$-types in CODF are dense over codes, and Theorem \ref{thm:strongdensityoftypes} that all definable $n$-types in CODF are $\dd$-dense over codes, where $\dd$ is the $\delta$-dimension (see later Section \ref{delta-dim}). The following Proposition corresponds to the reduction to $n=1$ and is very much inspired by \cite{johnsonthesis}.

\begin{proposition}\label{lem:1-to-n} 
	Let $T$ be a complete theory. 
	If definable $1$-types are dense over codes (resp. over $a$-codes) 
	then definable $n$-types are dense over codes (resp. $a$-codes). 
	Moreover, assuming $T$ admits a code-definable fibered dimension function $d$, 
	if $1$-types are $d$-dense over codes (resp. $a$-codes) 
	then so are all definable $n$-types. 
\end{proposition}

A proof for the first part of the proposition can be found in \cite[Claim 6.3.2]{johnsonthesis}. We include the proof for the reader's convenience.  

\begin{proof}[Proof of Proposition \ref{lem:1-to-n}:] 
	Let us first show the argument for types without the dimension assumption. 
	We proceed by induction on $n\geq 1$, the base case being given by assumption. 
	Let $X\subseteq K^{n+1}$ be a definable set. 
	By induction, let $a\in\pi_n(X)(\U)$ be such that $tp(a/K)$ is $c(X)$-definable 
	i.e. $a\fork[c(X)]K$. Let $K'$ be an elementary extension of $K$ containing $a$ and $c(X)$. 
	By the case $n=1$, let $b\in X_a(\U)$ be such that $tp(b/K')$ is $c(X_a)$-definable, 
	so in particular $(c(X)\cup\{a\})$-definable. 
	Then we have both $a\fork[c(X)]K$ and $b\fork[c(X)a]K$, 
	so by Lemma \ref{lem:transitivity}, we have $(a,b)\fork[c(X)]K$, 
	which shows that $tp(a,b/K)$ is $c(X)$-definable. 
	By assumption $(a,b)\in X(\U)$ which completes the proof of the first statement.

	Note that exactly the same proof of the first statement of 
	Proposition \ref{lem:1-to-n} works when replacing density over codes by density over $a$-codes. 
	One simply replaces $\fork$ by $\afork$ and notices that if $b\in X_a(\U)$ and $tp(b/K')$ is $acl^{eq}(c(X_a))$-definable, 
	then it is also $acl^{eq}(c(X)\cup\{a\})$-definable given that $acl^{eq}(c(X_a))\subseteq acl^{eq}(c(X)\cup\{a\})$. 
	This shows that $b\afork[c(X)a]K$ and the proof is completed using Lemma \ref{lem:transitivity} exactly as before.
\\

Let us now show the second statement. We only show the result for $d$-density over codes since, as in the previous case, the proof for $d$-density over $a$-codes is completely analogous. We proceed by induction on $n\geq 1$, the base case being given by assumption. 

Fix $i\in\{0,1\}$ such that $d(X)=d(X(i))+i$. Notice that this implies that $X(i)\neq\emptyset$. By induction, let $a\in X(i)(\U)$ be such that $tp(a/K)$ is $c(X(i))$-definable and $d(a/K)=d(X(i))$. Since $X(i)$ is $c(X)$-definable by condition (Dim 5), $tp(a/K)$ is also $c(X)$-definable. Let $K'$ be an elementary extension of $K$ containing $a$. By the case $n=1$, let $b\in X_a(\U)$ be such that $tp(b/K')$ is $c(X_a)$-definable and $d(b/K')=d(X_a)=i$. By Lemma \ref{lem:transitivity}, $tp(a,b/K)$ is $c(X)$-definable. It suffices to show that $d(a,b/K)=d(X(i))+i$. We split in two cases: 

\textit{Case 1:} If $i=1$, then by part $(4)$ of Lemma \ref{lem:dim}, $d(b/Ka)=1$. This implies, by Lemma \ref{lem:additivity}, that $d(a,b/K)=d(a/K)+1=d(X(1))+1=d(X)$. 

\textit{Case 2:} If $i=0$, since $d(X_a)=0$ and $X_a$ is $(K\cup\{a\})$-definable, $d(b/Ka)=0$. Again by Lemma \ref{lem:additivity}, we have that $d(a,b/K)=d(a/K)+0=d(X(0))+0=d(X)$. 
\end{proof}

\begin{proposition}\label{prop:strong-density-ominimal} Let $T$ be an o-minimal theory and let $\dim$ be its dimension function. Then all definable types are $\dim$-dense over codes.
 \end{proposition}
\begin{proof}
In view of Proposition \ref{lem:1-to-n}, it suffices to show that definable $1$-types $\dim$-dense over codes. Let $K$ be a model of $T$ and let $X\subseteq K$ be a definable set. By o-minimality, $X$ is either finite or there is an interval $(b_0,b_1)$ with $b_0\in K\cup\{\pm\infty\}$ which is maximally contained in $X$ with respect to inclusion. If $X$ is finite, let $a$ be its minimal element. Then $tp(a/K)$ is $c(X)$-definable and $\dim(a/K)=0=\dim(X)$. Otherwise, if $b_0=-\infty$, let $a\in \U$ be any element realizing the type $-\infty$ over $K$. Then we have that $a\in X(\U)$, $tp(a/K)$ is $\emptyset$-definable. Finally, if $b_0\in K$, let $a\in \U$ realize the type $b_0^+$, which is a $c(X)$-definable type  and again $a\in X(\U)$. In both cases, one easily cheks that $\dim(a/K)=1=\dim(X)$. 
\end{proof}

\subsection{The case $n=1$} In this section $\U$ denotes a monster model of CODF. We will show that definable 1-types are dense over codes. Recall that for $B\subseteq \U$ and $c\in \U$, $tp(a/B)$ stands for the type of $a$ over $B$ in the language $\cL_{or}$. 

\begin{lemma}\label{lem:otypes} 
	Let $a=(a_1,\ldots, a_n)\in \U^n$ 
	and $\dim(a/K)=n$. Let $\epsilon_0$ realise the type $0^+$ over $K(a)^{rc}$. 
	Then for every $b=(b_1,\ldots,b_n)\in \U^n$ such that
	$|a-b|<\epsilon_0$, we have that $tp(a/K)=tp(b/K)$. 
\end{lemma}

\begin{proof} 
Let $x=(x_1,\ldots,x_n)$, $\varphi(x,y)$ an $\cL_{or}$-formula and $c\in K^{|y|}$. Suppose that $\varphi(a,c)$ holds. 
By the dimension assumption, $a$ belongs to $Int(\varphi(\U,c))$ (which is not empty). 
Therefore, we have that 
\[
\U\models \exists\epsilon>0\ \  \forall x\ \  (|a-x|<\epsilon\to \varphi(x,c)). 
\]
Since $K(a)^{rc}$ is an $\cL_{or}$-elementary substructure of $\U$, 
let $\epsilon\in K(a)^{rc}$ satisfy the above condition. 
By assumption $|a-b|<\epsilon_0<\epsilon$, hence $\varphi(b,c)$. 
\end{proof}

\begin{lemma}\label{lem:extep2} 
	Let $a\in \U$ be such that $\dim(\bar{\delta}^{n-1}(a)/K)=n$. 
	Then there is $b\in \U$ such that $tp(\bar{\delta}^{n-1}(b)/K)=tp(\bar{\delta}^{n-1}(a)/K)$ 
	and for all $k\geqslant n$, $\delta^k(b)=0$.
\end{lemma}
\begin{proof}
	Let $\epsilon_0\in \U$ realise $0^+$ over $K(\bar{\delta}^{n-1}(a))^{rc}$.
	By Lemma \ref{constants}, there is $b_{n}\in \U$ such that
	$$\delta(b_{n})=0 \text{ and }|b_{n}-\delta^{n-1}(a)|<\epsilon_0.$$
	
 Iterating this argument, applying again Lemma \ref{constants}, we get $b_i\in \U$, $1\leq i<n$, such that
	$$\delta(b_{i})=b_{i+1} \text{\ and \ }|b_{i}-\delta^{i}(a)|<\epsilon_0.$$

	Therefore by Lemma \ref{lem:otypes}, $tp(b_1,\dots, b_{n}/K)=tp(\bar{\delta}^{n-1}(a)/K)$.
	Since $(b_1,\dots, b_{n})=\bar \delta^{n-1}(b_1)$, the proof is completed.
\end{proof}

\begin{proposition}\label{prop:strong-1-density} In CODF, definable 1-types are dense over codes. 
\end{proposition}

\begin{proof}  Let $X\subseteq K$ be a definable set. Without loss of generality, we may assume that $X$ is infinite. Let $X^{\circledast}\subseteq K^n$ be the $\cL_{or}$-definable set constructed in Lemma \ref{lem:density} and which is such that 
	$X^\nabla$ is contained and dense in $X^{\circledast}$. It follows that $X^\nabla$ is also dense in $\overline{X^{\circledast}}$  and this property can be expressed as follows: for $x=(x_0,\ldots,x_{n-1})$ 
\begin{equation}\label{eq:1}\tag{E1}
\U\models x\in \overline{X^{\circledast}}\leftrightarrow (\forall \epsilon>0)( \exists z\in X)[|\bar{\delta}^{n-1}(z)-x|<\epsilon].
\end{equation}
This shows that $\overline{X^{\circledast}}$ is $c_\delta(X)$-definable. 

Let $\dim(\overline{X^{\circledast}})=\ell\leq n$. Notice that $\ell\geqslant 1$ as $X$ is infinite. Define $Y:=Int(\pi_\ell(\overline{X^{\circledast}}))$.  By Lemma \ref{lem:algind}, $Y\neq \emptyset$. Moreover, $Y$ is $c_\delta(X)$-definable. 

By Proposition \ref{prop:strong-density-ominimal}, 
let $u=(u_0,u_1\ldots,u_{n-1})\in \overline{X^{\circledast}}$ be such that for $\hat{u}=(u_0,\ldots,u_{\ell-1})$ 
we have that $\hat{u}\in Y$, $\dim(\hat{u}/K)=l$ and $tp(\hat{u}/K)$ is $c(Y)$-definable. 
Let $\epsilon_0$ realise the type $0^+$ over $K(u)^{rc}$. Since $u\in \overline{X^{\circledast}}$, 
by (\ref{eq:1}) applied to $\epsilon_0$, let $a\in X$ be such that $|\bar{\delta}^n(a)-u|<\epsilon_0$. 
By Lemma \ref{lem:otypes}, we have that $tp(\bar{\delta}^{\ell-1}(a)/K)=tp(\hat{u}/K)$. 
Therefore, $tp(\bar{\delta}^{\ell-1}(a)/K)$ is  
$c(Y)$-definable, so in particular $c_\delta(X)$-definable. We split in cases. 

If $\ell<n$, then since $\dim(Y)=\ell$ we have that for all $x\in X$, 
\begin{equation}\label{eq:2}\tag{E2}
	\text{$\delta^{k}(x)\in dcl^{eq}(c(Y),\bar{\delta}^{\ell-1}(x))$ for all $k\geq \ell$}. 
\end{equation}
Therefore, by Lemma \ref{lem:dcltype} and (\ref{eq:2}) we have that $tp(\bar{\delta}^{k}(a)/K)$ is $c_\delta(X)$-definable for every $k\geqslant 0$. This implies that $tp_\delta(a/K)$ is $c_\delta(X)$-definable by Lemma \ref{forgetful}. 

Now suppose that $\ell=n$. Then $\dim(\bar{\delta}^{n-1}(a)/K)=n$, so by Lemma \ref{lem:extep2} there is $b\in \U$ such that $tp(\bar{\delta}^{n-1}(b)/K)=tp(\bar{\delta}^{n-1}(a)/K)$ and $\delta^k(b)=0$ for all $k\geqslant n$. This implies that $\bar{\delta}^{n-1}(b)\in X^{\circledast}$ and therefore $b\in X$ (by Lemma \ref{lem:density}). Since $tp(\bar{\delta}^{n-1}(a)/K)$ is $c(Y)$-definable, by Lemma \ref{lem:dcltype}, $tp(\bar{\delta}^{k}(b)/K)$ is also $c(Y)$-definable for all $k\geqslant 0$.
Therefore by Lemma \ref{forgetful}, $tp_\delta(b/K)$ is $c(Y)$-definable and hence, in particular, $c_\delta(X)$-definable. 

\end{proof}

\begin{proof}[Proof of Theorem \ref{thm:densityoftypes}] The theorem follows from Proposition \ref{prop:strong-1-density} and Proposition \ref{lem:1-to-n}. 
\end{proof}

\section{Elimination of imaginaries}\label{sec:EI}

The former proof of elimination of imaginaries for CODF, given by the third author \cite{point2011}, 
uses the \emph{d\'emontage} of semialgebraic sets in real closed fields (a fine decomposition of semialgebraic sets), 
as well as elimination of imaginaries for RCF. 
We will also use that last result together with the following general criterion due to E. Hrushovski \cite{hrushovski2014}. The formulation below is taken from \cite{johnsonthesis}. 

\begin{theorem}[{\cite[Theorem 6.3.1]{johnsonthesis}}]\label{criterion}
Let $T$ be an $\cL$-theory, with home sort $\U$ (meaning we work in a monster model $\U$ such that $\U^{eq}= dcl^{eq}(\U)$).
Let $\G$ be some collection of sorts. Then $T$ has elimination of imaginaries in the sorts $\G$, if the following conditions all hold:
		\begin{enumerate}	
		\item For every non-empty definable set $X\subset \U^1$, there is an $acl^{eq}(c(X))$-definable type in $X$.
		\item Every definable type in $\U^n$ has a code in $\G$, that is, there is some (possibly infinite) tuple from $\G$ which is interdefinable with a code for $p$.
		\item Every finite set of finite tuples from $\G$ has a code in $G$. 
		That is, if $S$ is a finite set of finite tuples from $G$, 
		then $c(S)$ is interdefinable with a finite tuple from $\G$.
	\end{enumerate}
\end{theorem}

We will need the following characterizations of definable types over models in RCF and CODF. The characterization for RCF follows from a more general characterization of definable types in o-minimal structures due to Marker and Steinhorn \cite{marker-steinhorn1994}.

\begin{theorem}[{Marker-Steinhorn \cite[Theorem 4.1]{marker-steinhorn1994}}]\label{thm:MS} 
	Let $\U$ be a monster model of RCF, $K$ be a model and $a\in \U^n$. Then the type $tp(a/K)$ is $K$-definable  
	if and only if $K$ is Dedekind complete in $K(a)^{rc}$. 
 \end{theorem}

\begin{theorem}[{\cite[Proposition 3.6]{brouette2017}}]\label{thm:Brouette} Let $\U$ be a monster model of CODF, $K$ be a model and $a=(a_1,\ldots,a_n)\in \U^n$. 
Then, the type $tp_\delta(a/K)$ is $K$-definable if and only if $K$ is Dedekind complete in $K(\bar{\delta}(a_1),\ldots,\bar{\delta}(a_n))^{rc}$. 
\end{theorem}

The following lemma ensures condition (2) of the previous criterion for CODF. 

\begin{lemma}\label{lem:codesfortypes} 
	Let $\U$ be a monster model of CODF and
	$q\in S_n^{CODF}(\U)$ be a definable type. Then $q$ has a code in $\U$. 
\end{lemma}
\begin{proof}  
	Suppose that $q$ is definable over a model $K$. It suffices to show that $p:=q|K$ (the restriction of $q$ to $K$) has a code in $K$. 
	 Let $(a_1,\cdots,a_n)$ be a realization of $p$. By Theorem \ref{thm:Brouette}, 
	we have that $K$ is Dedekind complete in $K(\bar{\delta}(a_1),\ldots, \bar{\delta}(a_n))^{rc}$. 
	For each $m\geqslant1$, let 
\[
p_m:= tp(\bar{\delta}^m(a_1),\ldots, \bar{\delta}^m(a_n)/K).
\]  
By Theorem \ref{thm:MS}, each type $p_m$ is definable. By elimination of imaginaries in RCF, let $b_m$ be a code for $p_m$ in $K$. We show that $(b_m)_{m\geqslant 1}$ is a code for $p$. Let $\sigma\in Aut_{\cL_{\delta}}(\U)$. Since $p\vdash p_m$ for all $m$, one direction is straighforward (namely if $\sigma(p)=p$, then $\sigma(p_m)=p_m$ and so $\sigma(b_m)=b_m$). For the converse, suppose that $\sigma(b_m)=b_m$ for all $m\geqslant 1$. Then $\sigma(p_m)=p_m$ for all $m\geqslant 1$. Let us show that $\sigma(p)=p$.

Let $\varphi(x_1,\ldots, x_n,y_1,\cdots,y_m)$ be an $\cL_{\delta}$-formula (without parameters). By quantifier elimination in CODF, we may assume $\varphi$ is quantifier-free. Let $\varphi^\ast(\overline{x}_1, \ldots, \overline{x}_n,\overline{y_1},\cdots,\overline{y}_m)$ be its associated $\cL_{or}$-formula. Without loss of generality, and to ease notation, let $k$ be an integer such that $\ell(\overline{x_i})=\ell(\overline{y_j})=k$ for all $i\in\{1,\ldots,n\}$ and all $j\in\{1,\ldots,m\}$. For $c=(c_1,\ldots,c_m)\in \U^m$, let $\bar{\delta}^k(c)$ denote the tuple $(\bar{\delta}^k(c_1),\ldots,\bar{\delta}^k(c_m))$. For any tuple $c\in K^m$ we have that:

\begin{align*}
	\U\models \varphi(a_1,\ldots, a_n,c) 	& \Leftrightarrow  \U\models \varphi^\ast(\bar{\delta}^{k}(a_1),\ldots, \bar{\delta}^{k}(a_n),\bar{\delta}^{k}(c))\\
																	& \Leftrightarrow  \U\models \varphi^\ast(\bar{\delta}^{k}(\sigma(a_1)),\ldots, \bar{\delta}^{k}(\sigma(a_n)),\bar{\delta}^{k}(c))\\
																	& \Leftrightarrow  \U\models \varphi(\sigma(a_1),\ldots, \sigma(a_n), c) \\ 
\end{align*}
Indeed, the first and third equivalences hold by definition of $\varphi^*$ while the second equivalence holds since $\sigma(p_k)=p_k$ implies that ($\varphi^\ast\in p_k\;\implies \varphi^\ast\in \sigma(p_k))$.
\end{proof}

\begin{theorem} \label{ei} 
	CODF has elimination of imaginaries. 
\end{theorem}
\begin{proof} This follows by Theorem \ref{criterion}. Condition (1) is implied by Theorem \ref{thm:densityoftypes}, condition (2) corresponds to Lemma \ref{lem:codesfortypes} and condition (3) is straightforward since the main sort is a field (so weak elimination of imaginaries implies elimination of imaginaries \cite[Fact 5.5]{messmer1996}). 
\end{proof}

\begin{remark} It is not clear how to extend the previous theorem to general topological fields with generic derivations as defined in \cite{guzy-point2010} even assuming that definable types are dense over codes (or over $a$-codes). The main obstacle is to show an analog of Lemma \ref{lem:codesfortypes} in the general context. If one tries to mimic the proof given for CODF, one runs into the problem of showing that given a type $tp_\delta(a/K)$, the types $tp(\bar{\delta}^n(a)/K)$ (in the ring language) are definable. Here we used the characterization of definable types in CODF given in Theorem \ref{thm:Brouette}. A substitute to Theorem \ref{thm:Brouette} can be found in a recent paper of Rideau and Simon \cite{rideau17}. Let us recall the precise setting they are working in. Let $T$ be a NIP $\cL$-theory admitting elimination of imaginaries and let $\tilde T$ be a complete $\tilde{L}$-enrichment of $T$. Let $N\models \tilde T$ and let $A=dcl_{\tilde{\cL}}^{eq}(A)\subset N_{\tilde{L}}^{eq}$. In \cite[Corollary 1.7]{rideau17}, the authors provide a sufficient condition on $\tilde T$ implying that the underlying $\cL$-type of any $\tilde{\cL}^{eq}(A)$-definable type over $N$ is $A\cap \mathcal R$-definable, where $\mathcal R$ denotes the set of all $\cL$-sorts. In the case of $T=RCF$ and $\tilde T=CODF$, we get that such $\cL$-type is $A\cap N$-definable. 
However, their sufficient condition requires the existence of a model $M$ of $\tilde T$ whose $\cL$-reduct is uniformly stably embedded in every elementary extension. Somehow surprisingly, in the case of $\tilde T=CODF$ and $T=RCF$, it has been shown that there is such a model. Indeed, building on a previous proof that CODF has archimedean models, the first author showed in \cite{brouettethesis} that $\R$ can be endowed with a derivation in such a way that its expansion becomes a model of CODF. By the above theorem of Marker and Steinhorn (Theorem \ref{thm:MS}), $\R$ (viewed as an ordered field) is uniformly stably embedded in every elementary extension. 
\end{remark}

\section{$\delta$-dimension and definable types in CODF\label{delta-dim}}

The fact that closed ordered differential fields may be endowed with a fibered dimension was first proven in \cite{BMR}. Their definition relies on a cell decomposition theorem in CODF. Alternatively, one can obtain the same dimension function proceeding as in \cite{vandendries1989}, as was done in \cite{GP12} working in a broader differential setting (see \cite[Proposition 2.13]{GP12}). Here we will call such a dimension function the \emph{$\delta$-dimension} and denote it by $\dd$ (in place of $t$-$\dim$, the notation used in \cite{GP12}). Throughout this section let $\U$ be a monster model of CODF and $K\subset \U$ be a small model. 

Instead of providing the original definition given in \cite{GP12} of $\dd$, we will use a characterization in terms of the following closure operator, also proven to be equivalent in \cite{GP12}. For $B\subseteq \U$ and $a\in \U$, say that $a\in cl(B)$ if and only if there is a differential polynomial $q\in \Q\langle B\rangle\{x\}\setminus\{0\}$ such that $q(a)=0$. This closure operator $cl$ defines a pre-geometry on $\U$ (see \cite[Lemma 2.8]{GP12}). The natural notion of relative independence induced by $cl$ is defined as follows: a set $A$ is $cl$-independent over $B$ if and only if for every $a\in A$ one has that 
\[
a\notin cl(B\cup (A\setminus\{a\})).  
\] 
Finally, define  
\begin{equation}\label{eq:3}\tag{E3}
\text{cl-$\dim$}(a/B) := \max \{ |C| : C\subseteq \Q\langle Ba\rangle, \text{ $C$ is $cl$-independent over $B$} \}.  
\end{equation}

Let $a\in \U^n$, $b\in \U^m$ and $B\subseteq \U$. Notice that $\dc(a/B)\leq n$ and that $\dc(a/B)\leq \dc(ab/B)$. Moreover, $\dc$ is additive, that is,  
\[
\dc(ab/B)=\dc(a/B)+\dc(b/Ba). 
\]

The following proposition characterizes the $\delta$-dimension in terms of $cl$ and gathers a second property that we will later need. 

\begin{proposition}[Guzy-Point]\label{prop:deltadim} Let $X\subseteq K^n$ be a definable set. Then
\begin{enumerate}
\item $\dd(X)=\max\{\dc(a/K) : a\in X(\U)\}$, 
\item if $n=1$, $\varphi(x)$ is a quantifier-free $\cL_\delta$-formula that defines $X$, then $\dd(X)=1$ if and only $\varphi^*(\U)$ contains a non-empty open set.  
\end{enumerate}
\end{proposition} 

\begin{proof}
Part (2) corresponds to \cite[Lemma 2.11]{GP12} and part (1) is a special case of \cite[Proposition  3.12]{GP12}. 
\end{proof}

The following lemma addresses the somehow classical issue of showing that the function $\delta$-$\dim$ induced on types and tuples as in Subsection 1.3 (equation (\ref{eq:type-dimension})), coincides, as expected, with $\dc$.

\begin{lemma}\label{lem:deltadim} Let $a=(a_1,\ldots,a_n)\in \U$ and $p:=tp_\delta(a/K)$. Then 
\[
\dc(a/K)=\inf\{\dd(\varphi(\U)):\; \varphi(x)\in p\}.
\]

\end{lemma}
\begin{proof} 
	Set $m:=\inf\{\delta$-dim$(\varphi(\U)):\; \varphi(x)\in p\}\leqslant n$ and let $\varphi\in p$ be such that $\delta$-$\dim(\varphi(\U))=m$. Since $\U\models \varphi(a)$, by Proposition \ref{prop:deltadim} we have that 
	\[
	\text{$cl$-$\dim$}(a/K)\leqslant \max\{\text{$cl$-$\dim$}(b/K) : b\in \varphi(\U)\}= \text{$\delta$-$\dim$}(\varphi(\U))=m. 
	\]
	We show that $m\leqslant \text{$cl$-$\dim$}(a/K)$. Suppose that $cl$-dim$(a/K)=d\leq n$. Let $b\in \U^d$ belong to $K\langle a\rangle$ such that $cl$-dim$(b/K)$=d=$cl$-dim$(a/K)$. So there are $q_i\in K\langle b\rangle\{x_i\}\setminus\{0\}$, $1\leq i\leq n$, be such that $q_i(a_i)=0$. Rewrite $q_i$ as differential polynomials $p_i\in K\langle y,x_i\rangle$. The formula $\psi(x_1,\ldots,x_n):=\exists y\;\bigwedge_{i=1} ^n\;(p_i(y,x_i)=0\wedge \exists u\,p_i(y,u)\neq 0)$ is in $p$. Let $c\in \U^{n}$ be any tuple such that $\psi(c)$ holds. Then there exists $z\in \U^d$ such that $\bigwedge_{i=1}^n(p_i(z,c_i)=0\wedge \exists u p_i(z,u)\neq 0)$. This shows that $\dc(c/Kz)=0$. Notice that since $z$ has length $d$, we have that $\dc(z/K)\leq d$. Therefore, we have that 
	\[
	\dc(c/K)\leqslant \dc(cz/K)= \dc(z/K)+\dc(c/Kz)\leqslant d.
	\]
	By Proposition \ref{prop:deltadim}, $\dd(\psi(\U))\leq d$, which shows $m\leqslant d$.
\end{proof}

Now let us show that the dimension function $\delta$-$\dim$ is code-definable. 

\

\begin{lemma}\label{lem:deltadimcode} 
	The $\delta$-dimension is code-definable.
\end{lemma}
\begin{proof} 
	Let $X\subseteq \U^{n+1}$ be a definable set. It suffices to show that $X(1)$ is $c_\delta(X)$-definable. 
	Suppose that the quantifier-free formula $\varphi(x,y,c)$ defines $X$ where $x=(x_1,\ldots,x_n)$, $y$ is a single variable and $c=(c_1,\ldots, c_s) \in \U^s$ are all the parameters in the formula. 
	Let $\sigma\in Aut_{\cL_\delta}(\U)$ be such that $\sigma(X)=X$. We show that $\sigma(X(1))=X(1)$. Pick $a\in X(1)$, so by definition $a\in \pi_m(X)$ and $\dd(X_a)=1$. 
	Since $\sigma(\pi_m(X))=\pi_m(X)$, we have that $\sigma(a)\in \pi_m(X)$. It remains to show that $\dd(X_{\sigma(a)})=1$. 

	Consider the $\cL_{or}$-formula $\varphi^*(\overline{x_1},\ldots,\overline{x_n}, \overline{y},\bar{\delta}^{t_1}(c_1),\ldots,\bar{\delta}^{t_s}(c_s))$ and let $m_i:=\ell(\overline{x_i})$ for $i\in\{1,\ldots,n\}$ and $\ell:=\ell(\overline{y})$. Recall that we have that $\varphi(x,y,c)$ is equivalent to 
	\[
	\varphi^*(\bar{\delta}^{m_1}(x_1),\dots,\bar{\delta}^{m_n}(x_n),\bar{\delta}^{\ell}(y),\bar{\delta}^{t_1}(c_1),\ldots,\bar{\delta}^{t_s}(c_s)).
	\]
	To ease notations in this proof, we will denote by $\bar \delta(a)$ the tuple $(\bar\delta^{m_1}(a_1),\cdots,\bar\delta^{m_n}(a_n))$, and by $\bar{\delta}(c)$ the tuple $(\bar{\delta}^{t_1}(c_1),\ldots,\bar{\delta}^{t_s}(c_s))$. 

	By Proposition \ref{prop:deltadim} (part 2), we have that 
	\begin{align*}
	\dd(X_{a})=1		& \Leftrightarrow \text{$\varphi^*(\bar{\delta}(a),\U,\bar{\delta}(c))$ contains a non-empty open set}\\		
								& \Leftrightarrow \U \models \exists \overline{z}\ \exists \epsilon>0\ \forall \overline{w}\ [ |\overline{w}-\overline{z}|<\epsilon \to \varphi^{\ast}(\bar\delta(a),\overline{w},\bar{\delta}(c))]\\
							 	& \Leftrightarrow \U \models \exists \overline{z}\ \exists \epsilon>0\  \forall \overline{w}\ [ |\overline{w}-\overline{z}|<\epsilon \to \varphi^{\ast}(\bar\delta(\sigma(a)),\overline{w},\bar{\delta}(\sigma(c)))] \\
							 	& \Leftrightarrow \text{$\varphi^*(\bar{\delta}(\sigma(a)),\U,\bar{\delta}(\sigma(c)))$ contains a non-empty open set}\\
							 	& \Leftrightarrow \dd(X_{\sigma(a)})=1.
	\end{align*}
\end{proof}

Before proving Theorem \ref{thm:strongdensityoftypes}, 
we need the following technical lemma which will ensure that we can find an element of the required dimension in a definable subset.

\begin{lemma}\label{lem:compacity} 
Let $a\in \U$ be such that $dim(\bar{\delta}^{n-1}(a)/K)=n$ and $x$ a variable of length 1. Let $\Sigma$ be the set of $\cL_{\delta}$-formulas $\Sigma:=\Sigma_1\cup \Sigma_2\cup \Sigma_3$ with parameters in $K$ where  
\begin{equation*}\label{eq:4}
\begin{array}{l}
\Sigma_1:=\{q(x)\neq 0 \colon q\in K\{x\}\setminus\{0\}\} \\
\Sigma_2:=\{\varphi(\bar{\delta}^{n-1}(x)) : \varphi\in tp(\bar{\delta}^{n-1}(a)/K) \}  \\ 
\Sigma_3:=\bigcup_{i\geqslant 1} \{\delta^{n+i}(x)< c : c\in K(\bar{\delta}^{n+i-1}(x))\}\\
\end{array}
\end{equation*}
Then $\Sigma$ is consistent and can be completed into a unique type $p\in S_1^{CODF}(K)$. 
\end{lemma}
\begin{proof} By quantifier elimination, we may assume that all formulas in $\Sigma$ are quantifier-free. By compactness, it suffices to show that every finite subset $\Theta$ of $\Sigma(x)$ is consistent (that $\Sigma$ can be completed into a unique type will follow from the fact that CODF admits quantifier elimination). 
	Let $m\geqslant 0$ be an integer such that
\begin{equation*}\label{eq:5}
	\begin{array}{l}
		\Theta\cap \Sigma_1\subseteq \{q(\bar{\delta}^m(x))\neq 0 \colon q\in K[\bar{\delta}^m(x)]\setminus\{0\}\} \\
		\Theta\cap \Sigma_3\subseteq \bigcup_{1<i\leqslant m} \{\delta^{n+i}(x)< c : c\in K(\bar{\delta}^{n+i-1}(x))\}.\\
	\end{array}
\end{equation*}
By multiplying all polynomials appearing in $\Theta\cap\Sigma_1$, we may assume $\Theta\cap\Sigma_1=\{q(x)\neq 0\}$ for some $q\in K[\bar{\delta}^m(x)]\setminus\{0\}$. 
Similarly, by replacing all formulas in $\Theta\cap\Sigma_2(x)$ by their conjunction, 
we may suppose $\Theta\cap\Sigma_2=\{\varphi\}$. Therefore, $\Theta$ is consistent if and only if the following formula $\theta(x)$ is consistent 
\begin{equation*}\label{eq:5}
\theta(x):=q(\bar{\delta}^m(x))\neq 0 \wedge \varphi(\bar{\delta}^{n-1}(x)) \wedge \psi(\bar{\delta}^{n+m}(x)) 
\end{equation*}
where $\psi(\bar{\delta}^{n+m}(x))$ corresponds to the conjunction of all formulas in $\Theta\cap\Sigma_3$. Suppose that 
\[
\psi(\bar{\delta}^{n+m}(x)):=\bigwedge_{i=1}^m \bigwedge_{j=1}^{m_i} \delta^{n+i}(x)<f_{ij}(\bar{\delta}^{n+i-1}(x)), 
\] 
with $m_i$ a positive integer for $i\in\{1,\ldots,m\}$ and $f_{ij}$ a rational function over $K$ in $n+i-1$ variables. Since $dim(\bar{\delta}^{n-1}(a)/K)=n$, we may further suppose that $\varphi$ defines an open subset of $K^n$. 
Consider the $\cL_{or}$-definable set $\theta^*(K)\subseteq K^{n+m}$. 
By Lemma \ref{constants}, to show that $\theta$ is consistent it suffices to show that $\theta^*(K)$ is open and non-empty. For $\bar{x}=(x_0,\ldots,x_m)$ we have that 
\begin{equation*}\label{eq:5}
\theta^*(\bar{x}):=q^*(\bar{x})\neq 0 \wedge \varphi(x_0,\ldots,x_{n-1}) \wedge \psi^*(\bar{x}), 
\end{equation*}
where $q^*$ is a polynomial over $K$ in $m+1$ variables. The formula $\varphi$, seen as a formula over the variables $\bar{x}$, defines an open subset of the form $U\times K^m$ where $U$ is a non-empty open subset of $K^n$. The formula $\psi^*(\bar{x})$ defines an open subset of the form $K^n\times V$ for $V$ a non-empty open subset of $K^m$. Therefore $\varphi\wedge \psi^*$ defines the non-empty open subset $U\times V$ of $K^{n+m}$. Finally, the intersection of a Zariski open subset of $K^{n+m}$ with any non-empty open subset of $K^{n+m}$ (in particular with $U\times V$) is non-empty and remains open. This shows that $\theta^*(K)$ is a non-empty open subset of $K^{n+m}$. 
\end{proof}

We have all the elements to show Theorem \ref{thm:strongdensityoftypes}. 

\begin{proof}[Proof of Theorem \ref{thm:strongdensityoftypes}:] By Lemma \ref{lem:deltadimcode} and Proposition \ref{lem:1-to-n}, it suffices to prove the property for non-empty definable subsets $X\subseteq K$. By Proposition \ref{prop:strong-1-density}, there is $a\in X(\U)$ such that $tp(a/K)$ is $c_\delta(X)$ definable. We split in cases depending on the value of $\delta$-$\dim(X)$. 

Suppose first that $\delta$-$\dim(X)=0$. Since $X$ is $K$-definable, by Lemma \ref{lem:deltadim}, we must already have that $\delta$-$\dim(a/K)=0$. 

Suppose now that $\delta$-$\dim(X)=1$. By Lemma \ref{lem:density}, there is an $\cL_{or}$-definable subset $X^{\circledast}\subset K^n$ 
such that  
\begin{equation}\label{eq:1}\tag{E1}
\U\models x\in \overline{X^{\circledast}}\leftrightarrow (\forall \epsilon>0)( \exists z\in X)[|\bar{\delta}^{n-1}(z)-x|<\epsilon].
\end{equation}
Since $a\in X$, we have that $\bar{\delta}^{n-1}(a)\in X^{\circledast}$. 
Moreover, since $\delta$-$\dim(X)=1$, we must have that $\dim(\bar{\delta}^{n-1}(a)/K)=n$. We need to find an element $b\in X(\U)$ such that $tp(\delta^{n-1}(a)/K)=tp(\delta^{n-1}(b)/K)$, $\delta$-$\dim(b/K)=1$ and $tp_\delta(b/K)$ is still $c_\delta(X)$-definable. %
We will use Lemma \ref{lem:compacity} to find such an element. Note that the element $a$ which we found in Proposition \ref{prop:strong-1-density} is always such that $\delta$-$\dim(a/K)=0$. 

Let
\begin{equation*}\label{eq:4} 
\begin{array}{l}
\Sigma_1:=\{q(x)\neq 0 \colon q\in K\{x\}\} \\
\Sigma_2:=\{\varphi(\bar{\delta}^{n-1}(x)) : \varphi\in tp(\bar{\delta}^{n-1}(a)/K) \}  \\ 
\Sigma_3:=\bigcup_{i\geqslant 1} \{\delta^{n+i}(x)< c : c\in K(\bar{\delta}^{n+i-1}(x))\},\\
\end{array}
\end{equation*} 
by Lemma \ref{lem:compacity}, $\Sigma_1\cup\Sigma_2\cup\Sigma_3$ determines a unique type $p$  and we will show that it has all the required properties. Let $b$ be a realization of $p$. Since $tp(\bar{\delta}^{n-1}(b)/K)=tp(\bar{\delta}^{n-1}(a)/K)$, by (\ref{eq:1}) we have that $b\in X(\U)$. The fact that it satisfies $\Sigma_1$ ensures that $\delta$-$\dim(b/K)=1$. It remains to show that $tp_{\delta}(b/K)$ is $c_{\delta}(X)$-definable. We will show that for every $m$, $tp(\bar \delta^{m}(b)/K)$ is $c_\delta(X)$-definable. For $m<n$, it follows from the fact that $tp(\bar \delta^{n-1}(a)/K)$ is $c_\delta(X)$-definable. By induction on $i\geq 0$, let us show that $tp(\bar \delta^{n+i}(b)/K)$ is $c_\delta(X)$-definable. By Lemma \ref{transitivity}, it suffices to show that $tp(\delta^{n+i}(b)/K\bar \delta^{n-1+i}(b))$ is $c_\delta(X)$-definable. 
Actually, $tp(\delta^{n+i}(b)/K\bar \delta^{n-1+i}(b))$ is even $\emptyset$-definable because it is the type $-\infty$ over $K(\bar \delta^{n-1+i}(b))$.

\

Since $tp(\bar{\delta}^{n-1}(a)/K)$ is $c_{\delta}(X)$-definable, by Lemma \ref{transitivity}, $tp(\bar{\delta}^{k}(b)/K)$ is also $c_{\delta}(X)$-definable for all $k\geqslant 0$. Therefore, $tp_\delta(b/K)$ is $c_\delta(X)$-definable by Lemma \ref{forgetful}. 
\end{proof}

\bibliographystyle{alpha}
\bibliography{../biblio}

\begin{thebibliography}{{Hru}14}

\bibitem[BMR09]{BMR}
Thomas Brihaye, Christian Michaux, and C\'edric Rivi\`ere.
\newblock Cell decomposition and dimension function in the theory of closed
  ordered differential fields.
\newblock {\em Ann. Pure Appl. Logic}, 159(1-2):111--128, 2009.

\bibitem[Bro15]{brouettethesis}
Quentin Brouette.
\newblock {\em Differential algebra, ordered fields and model theory}.
\newblock PhD thesis, Universit\'e de Mons, Belgium, 2015.

\bibitem[Bro17]{brouette2017}
Quentin Brouette.
\newblock Definable types in the theory of closed ordered differential fields.
\newblock {\em Archive for Mathematical Logic}, 56(1):119--129, 2017.

\bibitem[GP10]{guzy-point2010}
Nicolas Guzy and Fran\c{c}oise Point.
\newblock Topological differential fields.
\newblock {\em Annals of Pure and Applied Logic}, 161(4):570 -- 598, 2010.

\bibitem[GP12]{GP12}
Nicolas Guzy and Fran\c{c}oise Point.
\newblock Topological differential fields and dimension functions.
\newblock {\em J. Symbolic Logic}, 77(4):1147--1164, 2012.

\bibitem[HKR16]{HKR2016}
Martin Hils, Moshe Kamensky, and Silvain Rideau.
\newblock Imaginaries in separably closed valued fields.
\newblock arXiv:1612.02142 [math.LO], 2016.

\bibitem[{Hru}14]{hrushovski2014}
Ehud {Hrushovski}.
\newblock Imaginaries and definable types in algebraically closed valued
  fields.
\newblock In Franz-Viktor~Kuhlmann Antonio~Campillo and Bernard Teissier,
  editors, {\em Valuation theory in interaction}, Series of Congress Reports of
  the EMS, pages 297--319. EMS, 2014.

\bibitem[Joh16]{johnsonthesis}
Will Johnson.
\newblock {\em Fun with fields}.
\newblock PhDthesis, University of California, Berkeley, 2016.

\bibitem[Mar96]{marker1996}
David Marker.
\newblock {\em Chapter 2: Model Theory of Differential Fields}, volume~5 of
  {\em Lecture Notes in Logic}, pages 38--113.
\newblock Springer-Verlag, Berlin, 1996.

\bibitem[Mes96]{messmer1996}
Margit Messmer.
\newblock {\em Chapter 4: Some Model Theory of Separably Closed Fields},
  volume~5 of {\em Lecture Notes in Logic}, pages 135--152.
\newblock Springer-Verlag, Berlin, 1996.

\bibitem[MS94]{marker-steinhorn1994}
David Marker and Charles~I. Steinhorn.
\newblock Definable types in o-minimal theories.
\newblock {\em The Journal of Symbolic Logic}, 59(1):185--198, 1994.

\bibitem[Poi11]{point2011}
Fran\c{c}oise Point.
\newblock Ensembles d\'efinissables dans les corps ordonn\'es
  diff\'erentiellement clos.
\newblock {\em C. R. Math. Acad. Sci. Paris}, 349(17-18):929--933, 2011.

\bibitem[Rid14]{rideauthesis}
Silvain Rideau.
\newblock {\em \'Eliminations dans les corps valu\'es}.
\newblock PhD thesis, Universit\'e Paris-Sud, France, 2014.

\bibitem[RS17]{rideau17}
Silvain Rideau and Pierre Simon.
\newblock Definable and invariant types in enrichments of {NIP} theories.
\newblock {\em J. Symb. Log.}, 82(1):317--324, 2017.

\bibitem[Sin78]{singer1978}
Michael~F. Singer.
\newblock The model theory of ordered differential fields.
\newblock {\em J. Symbolic Logic}, 43(1):82--91, 1978.

\bibitem[SS14]{simon2014}
Pierre Simon and Sergei Starchenko.
\newblock On forking and definability of types in some {DP}-minimal theories.
\newblock {\em J. Symb. Log.}, 79(4):1020--1024, 2014.

\bibitem[TZ12]{tent-ziegler2012}
Katrin Tent and Martin Ziegler.
\newblock {\em A course in model theory}, volume~40 of {\em Lecture Notes in
  Logic}.
\newblock Association for Symbolic Logic, La Jolla, CA; Cambridge University
  Press, Cambridge, 2012.

\bibitem[vdD89]{vandendries1989}
Lou van~den Dries.
\newblock Dimension of definable sets, algebraic boundedness and henselian
  fields.
\newblock {\em Annals of Pure and Applied Logic}, 45(2):189 -- 209, 1989.

\bibitem[vdD98]{vandendries1998}
Lou van~den Dries.
\newblock {\em Tame topology and o-minimal structures}, volume 248 of {\em
  London Mathematical Society Lecture Note Series}.
\newblock Cambridge University Press, Cambridge, 1998.

\end{thebibliography}

\end{document}